\documentclass[reqno]{amsart}

\usepackage{enumitem}
\usepackage{amssymb}
\usepackage{hyperref}
\usepackage{color}

\newcommand*{\mailto}[1]{\href{mailto:#1}{\nolinkurl{#1}}}

\newtheorem{theorem}{Theorem}[section]

\newtheorem{lemma}[theorem]{Lemma}

\newtheorem{proposition}[theorem]{Proposition}

\newcommand{\R}{{\mathbb R}}
\newcommand{\C}{{\mathbb C}}
\newcommand{\cL}{{\mathcal L}}
\newcommand{\cF}{{\mathcal F}}

\newcommand{\cR}{{\mathcal R}}
\newcommand{\cK}{{\mathcal K}}
\newcommand{\cV}{{\mathcal V}}
\newcommand{\cU}{{\mathcal U}}

\newcommand{\om}{\omega}
\newcommand{\ds}{\displaystyle}
\newcommand{\E}{\mathrm{e}}

\numberwithin{equation}{section}

\begin{document}

\title[On Dispersive Estimates for Klein--Gordon Equations]{On Dispersive Estimates for
One-Dimensional  Klein--Gordon Equations}

\author[E.\ Kopylova]{Elena Kopylova}
\address{Faculty of Mathematics\\ University of Vienna\\
Oskar-Morgenstern-Platz 1\\ 1090 Wien\\ Austria}
\email{\mailto{Elena.Kopylova@univie.ac.at}}
\urladdr{\url{http://www.mat.univie.ac.at/~ek/}}
\thanks{{\it Research supported by the Austrian Science Fund (FWF) under Grant No.\ P 34177-N}}


\begin{abstract}
We improve previous results on  dispersive  decay for  1D Klein-Gordon equation.
We develop a novel approach, which  allows us  to establish the decay in more strong norms 
and weaken the assumption on the potential.
\end{abstract}

\maketitle

%%%%%%%%%%%%%%%%%%%%%%%%%%%%%%%%%%%%%%%%%%%%%%%%%%%%%%%%%%%%%%%%%%%%%%%%%%
\section{Introduction}
%%%%%%%%%%%%%%%%%%%%%%%%%%%%%%%%%%%%%%%%%%%%%%%%%%%%%%%%%%%%%%%%%%%%%%%%%%
We are concerned with  one-dimensional   Klein--Gordon equation
\begin{equation} \label{KGE}
\ddot \psi(x,t)= K\psi:=(\partial^2_x-m^2+V(x)) \psi(x,t),\quad (x,t)\in\R^2,\quad m>0,
\end{equation}
with real  potential $V.$ In vector form equation \eqref{KGE} reads
\begin{equation} \label{KGEv}
i\dot \Psi(t)=\cK \Psi(t),
\end{equation}
where
\begin{equation} \label{H}
\Psi(t)=\begin{pmatrix}
  \psi(t)
  \\
  \dot\psi(t)
  \end{pmatrix},\quad
 \cK = i\begin{pmatrix}
  0                          &   1\\
 \Delta-m^2+V  &   0
\end{pmatrix}.
\end{equation}
We assume that $V(x)$ is a continuous  function, and
\begin{equation}\label{V}
    |V(x)|\le C\langle x\rangle^{-\beta},\quad\quad \langle x\rangle=(1+|x|^2)^{1/2},\quad x\in\R
\end{equation}
with some $\beta>0$.

Our goal is to prove dispersive decay estimates for these equations. This is a well-studied
area and  our main contribution is to improve previous results. 
To formulate them, we introduce the weighted  Sobolev spaces $W^{\ell,p}_{\sigma}$, $\sigma\in\R$, associated with the norm
\begin{equation*}
   \Vert \psi\Vert_{W^{\ell,p}_{\sigma}}=  \sum\limits_{k=0}^\ell\Vert\langle x\rangle^{\sigma}\psi^{(k)}\Vert_{L^{p}}, \quad 1\le p \le\infty,
 \end{equation*}
where $L^p=L^p(\R)$. Denote $L^p_\sigma=W^{0,p}_\sigma$,
$H^\ell_\sigma:=W^{\ell,2}_\sigma$, so $L^2_\sigma:=H^0_\sigma$.
Of course, the case $\sigma=0$ corresponds to the usual $W^{\ell,p}$ spaces without weight.
Denote
$$
{\cF}_{\sigma}=H^1_\sigma\oplus L^2_\sigma,\qquad \cL^1_{\sigma}=W^{1,1}_\sigma\oplus L^1_{\sigma},
\qquad \cL^{\infty}_{\sigma}=W^{1,\infty}_\sigma\oplus L^\infty_{\sigma}. 
$$
We recall (e.g., \cite{KK12} ) that for $V$ satisfies \eqref{V} with $\beta>1$, the spectrum of  $\cK$ consists of 
of a purely absolutely continuous part, covering $\Gamma:=(-\infty,-m)\cup (m,\infty)$, plus a finite number of eigenvalues located in $\R\setminus\Gamma$. 
In addition, there could be resonances at the edges $\pm m$ of the continuous spectrum.

Recall that the edges points $\lambda=\pm m$ are resonances if the equation $K\psi=-m^2\psi$ has nonzero solution
$\psi\in L^\infty_{-1/2-} (\R)\setminus L^2(\R)$ (cf. \cite{M}). In this case the operator $(K+m^2)^{-1}$ is unbounded.
In particular, the edges points $\lambda=\pm m$ are resonances for free Klein-Gordon equation with $V(x)\equiv 0$.

Our first result read as follows:
%%%%%%%%%%%%%%%%%%%%%%%%%%%%%%%%%%%%%%%%%%%%%%%%%%%%%%%%%%%%%%%%%%%%%%%%%%%%%%%%%
\begin{theorem}\label{Main}
Let \eqref{V} hold with $\beta>2$. Then the following decay holds
\begin{equation}\label{full}
\Vert \E^{-it\cK}P_c\Vert_{{\cF}_{\sigma}\to {\cF}_{-\sigma}}=\mathcal{O}(t^{-1/2}),\quad t\to\infty,\quad\sigma>1,
\end{equation}
where $P_c=P_c(\cK)$ is the orthogonal projection in
$L^2(\R)$ onto the continuous spectrum of $\cK$.
\end{theorem}
%%%%%%%%%%%%%%%%%%%%%%%%%%%%%%%%%%%%%%%%%%%%%%%%%%%%%%%%%%%%%%%%%%%%%%%%%%%%%%%%
Note that for the free Klein-Gordon  equation 
with $V=0$ the estimate \eqref{full} follow  from the explicit formula for the time evolution (see e.g. \cite{KK}).
Let us emphasize that we not require additional decay of $V$ for \eqref{full} in the case when edges of the continuous spectrum are resonances.
%%%%%%%%%%%%%%%%%%%%%%%%%%%%%%%%%%%%%%%%%%%%%%%%%%%%%%%%%%%%%%%%%%%%%%%%%%%%%%%%%
In the remaining result we restrict ourselves to non-resonance case.
\begin{theorem}\label{Main-new}
Let \eqref{V} hold with $\beta>3$. Then, in the non-resonant case, the following decay holds
\begin{equation}\label{full-new}
\Vert \E^{-it\cK}P_c\Vert_{{\cF}_{\sigma}\to {\cF}_{-\sigma}}=\mathcal{O}(t^{-3/2}),\quad t\to\infty,\quad \sigma>3/2.
\end{equation}
\end{theorem}
%%%%%%%%%%%%%%%%%%%%%%%%%%%%%%%%%%%%%%%%%%%%%%%%%%%%%%%%%%%%%%%%%%%%%%%%%%%%%%%%%%%%%%%%%%%%%55
The  decay (\ref{full-new}) was obtained in \cite {KK,K10}
under more restrictive condition  $\sigma>5/2$ and $\beta>5$ and under an additional assumption on the decay of $\nabla V(x)$. 
Now we improve this result to $\sigma>3/2$ assuming  (\ref{V}) with $\beta>3$ only.

The low energy   decay in $\cL^1\to \cL^\infty$ norms  was obtained in \cite {EKTM} (Theorem \ref{thKG1} below).
The approach of \cite {EKTM}  depends on the fact that the scattering matrix is in the Wiener algebra (i.e.\ its Fourier transform
is integrable). 
This result  implies the decay  (\ref{full})  and (\ref{full-new}) for the low energy part of solutions.
Unfortunately,  the decay  in $\cL^1\to \cL^\infty$ norms for the operator 
$\E^{-it\cK}P_c$ does not hold (even in the case $V=0$), in contrast to the Schr\"odinger operator. 
Namely,  for the free Schr\"odinger operator  the decay holds since  its integral kernel is bounded for $|t|\ge 1$ and  decays uniformly:
 \begin{equation}\label{full0}
\Vert \E^{-i\Delta t}(x,y)\Vert_{L^1\to L^{\infty}}
=\sup\limits_{x,y\in\R^3}\Big|\frac{\E^{i|x-y|^2/4t}}{\sqrt{4\pi it}}\Big|
\sim |t|^{-1/2},\quad |t|\ge 1.
\end{equation}
On the other hand, the Green function  $U(t,x,y)$ of the free Klein-Gordon equation does not decay for $|x-y|\sim |t|$. Namely, for $t>0$,
\begin{equation}\label{U}
 U(t,x,y)=\frac{1}{2}\theta(t-|x-y|)J_{0}(m\sqrt{t^{2}-|x-y|^{2}})\sim \frac 12,\quad  |x-y|\sim t.
\end{equation}

Here $J_{0}$ is the Bessel function of order 0, and $\theta$ is the Heavyside function.
This difference  reflects the distinct character of the wave propagation 
for the relativistic and nonrelativistic equations.
Namely, the  singularities of solutions  to the Schr\"odinger equations are concentrated at $t=0$ and disappear at infinity for $t\ne 0$ due to
infinite speed of propagation. On the other hand, in the case of  Klein-Gordon equation, the singularities move
with bounded velocities, thus they are present forever in the space.

As shown in  \cite {EKTM},  the decay  of solution $\psi(t)$ to \eqref{KGE} in $L^\infty$ norm  is possible only if  
the initial data $(\psi(0),\dot\psi(0)$ belong to generalized Sobolev space $W^{\frac 32,1}\oplus W^{\frac 12,1}$.

%However, those extra derivatives can be avoided if we work directly in  weighted energy norms. %as in (\ref{full})  and (\ref{full-new}).

Our approach relies on the  
Born expansion for high energy part of the dynamical group $\cU(t)=\cU(\zeta,t)=\E^{-it\cK}\zeta (\cK^2)$,  where 
$\zeta\in C^\infty(\R)$, $\zeta(\omega)=0$ for $\omega\le m^2+1$, and
$\zeta(\omega)=1$ for $\omega\geq m^2 + 2$,
\begin{eqnarray}\nonumber
\cU(t)&=&\cU_0(t)-i\int_0^t \cU_0(t-s){\mathcal V} \cU_0 (s)ds\\
\nonumber
&-&\int_0^t \cU_0(t-s){\mathcal V}\Big(\int_0^{s} \cU_0 (s-u){\mathcal V} \cU_0 (u)du\Big)ds+{\mathcal W}(t),
\end{eqnarray}
where $\cU_0(t)=\E^{-it\cK_0}\zeta (\cK_0^2)$, and 
$\cV$ is the matrix (\ref{cV}).

First we show that  $\Vert \cU_0(t)\Vert_{\cF_{\sigma}\to \cF_{-\sigma}}\le C(1+|t|)^{-\sigma}$ for any $\sigma>0$.
Then the  decay of the next two terms (under  appropriate values of $\beta$ in (\ref{V}))
follows by  standard estimates for the convolutions.
\smallskip\\
Finally, we prove the decay in $\cL^1\to \cL^\infty$ norms for remainder
%The  main difficulty is to obtain the  decay estimate of   the remainder $W(t)$, 
\begin{equation*}
{\mathcal W}(t)=\frac {i}{2\pi}\int\limits_{\Gamma} \!\E^{-i\omega t}
\big((\cR_0(\omega+\!i0) \cV)^3 \cR(\omega+\!i0)\!-\!(\cR _0(\omega-\! i0)\cV)^3 \cR(\omega-\! i0)\big)\zeta(\omega^2)d\omega, 
\end{equation*}
where $\Gamma=(-\infty,-m]\cup [m,\infty)$,   and $\cR$  and
$\cR_0$ stand for the  resolvents  of the operators $\cK$ and $\cK_0$, respectively.
For the proof we use  a representation of $\cR$ via Jost solutions  and their properties 
obtained in  \cite{EKTM}.
\smallskip\\
Note that dispersion estimates of type \eqref{full}--(\ref{full-new}) play an important role in proving asymptotic
stability of solitons and of scattering asymptotics in the associated one-dimensional nonlinear Klein-Gordon equations,
see for example \cite{KK11}.

%%%%%%%%%%%%%%%%%%%%%%%%%%%%%%%%%%%%%%%%%%%%%%%%%%%%%%%%%%%%%%%%%%%%%%%%%%%%%
\setcounter{equation}{0}
\section{Free Klein-Gordon equation}\label{KG0}
%%%%%%%%%%%%%%%%%%%%%%%%%%%%%%%%%%%%%%%%%%%%%%%%%%%%%%%%%%%%%%%%%%%%%%%%%%%%%
%%%%%%%%%%%%%%%%%%%%%%%%%%%%%%%%%%%%%%%%%%%%%%%%%%%%%%%%%%%%%%%%%%%%%%%%%%%%%
Here we consider the case $V= 0$. Denote
$$
 \cK_0 = i\begin{pmatrix}
  0                          &   1\\
 \Delta-m^2  &   0
\end{pmatrix}
$$
The resolvent $\cR_0(\omega)=(\cK_0-\omega)^{-1}$ of the operator   $\cK_0$  can be expressed in terms of the resolvent  
$R_0(\omega)=(-\partial^2_x-\omega)^{-1}$ of the Schr\"odinger operator as
\[
 \cR_0(\omega)=
   \begin{pmatrix}
 0            &  0\\
  -i          &  0
\end{pmatrix}
+
\begin{pmatrix}
 \omega &  i \\
  -i\omega^2&  \omega
\end{pmatrix} R_0(\omega^2-m^2).
\]
For the dynamical group $\E^{-it\cK_0}$ of the free Klein-Gordon equation the spectral representation  holds:
\begin{eqnarray}\nonumber
\!\!\!\!\!\!\!\!\!\!\!\!&&\!\!\!\!\!\!\!\!\!\!\!\!\!\!\!\!\E^{-it\cK_0}
   =\frac 1{2\pi i}\int\limits_\Gamma  \E^{-it\omega} ( \cR_0(\omega+i0)-\cR_0(\omega-i0))\,d\omega\\
\label{Rep0}
\!\!\!\!\!\!\!\!\!\!&&\!\!\!\!\!\!\!\!\!\!\!\!\!\!=\frac 1{2\pi i}\int\limits_\Gamma \E^{-i\om t}
\begin{pmatrix}
 \omega &  i \\\label{PP1}
  -i\omega^2&  \omega
\end{pmatrix}\left(R_0((\omega+i0)^2\!-\!m^2)-\!R_0((\omega-i0)^2\!-\!m^2)\right)d\om.
\end{eqnarray}
%where $\Gamma=(-\infty,-m)\cup (m,\infty)$.
Recall, that the integral kernel of   $R_0(\omega)$ reads
$$
R_0(\omega,x,y)=-\frac{e^{i\sqrt\omega|x-y|}}{2i\sqrt\omega}.
$$
Hence,
\begin{equation}\label{K0-rep}
[\E^{-it\cK_0}](x,y)=\frac{1}{2\pi}\int_{-\infty}^{\infty}{\mathcal M}_t(k)\,\E^{i|y-x|k}\,dk,
\end{equation}
where
\begin{equation}\label{M}
{\mathcal M}_t(k)=\begin{pmatrix}
 \cos(t\sqrt{k^2 + m^2})  & \ds\frac{\sin(t\sqrt{k^2 + m^2})}{\sqrt{k^2+m^2}}\\
  -\sqrt{k^2+m^2}\sin(t\sqrt{k^2 + m^2})&  \cos(t\sqrt{k^2 + m^2})
\end{pmatrix}.
\end{equation}
%%%%%%%%%%%%%%%%%%%%%%%%%%%%%%%%%%%%%%%%%%%%%%%%%%%%%%%%%%%%%%%%%%%%%
\begin{proposition}\label{pr1} 
\begin{equation}\label{K0-decay}
\Vert  e^{-i \cK_0t} \Vert_{{\cF}_{\sigma}\to {\cF}_{-\sigma}}
\le C (1+t)^{-1/2},\quad t\ge 0,\quad \sigma>1/2.
\end{equation}
\end{proposition}
%%%%%%%%%%
The proposition follows from two lemmas below.
\begin{lemma}\label{lem1} 
For  any $\chi\in C_0^\infty(\R)$ %with  bounded support  
the decay holds
\begin{equation}\label{WD2}
\Vert  e^{-i \cK_0t} \chi(\cK_0^2)\Vert_{\cL^1\to \cL^\infty}
\le  {\color{red}C} (1+t)^{-1/2}, \quad t\ge 0.
\end{equation}
\end{lemma}
\begin{lemma}\label{lem2} 
Let $\zeta$  be a smooth function such that $\zeta(\omega)=0$ for $\omega\le m^2+1$ and
$\zeta(\omega)=1$ for $\omega\geq m^2 + 2$. Then the decay holds
\begin{equation}\label{WD3}
\Vert  e^{-i \cK_0t} \zeta(\cK_0^2)\Vert_{{\cF}_{\sigma}\to {\cF}_{-\sigma}}
\le C (1+t)^{-\sigma},\quad\sigma>0,\quad t\ge 0.
\end{equation}
\end{lemma}

{\bf Proof of Lemma \ref{lem1}}
\\
Similarly to (\ref{K0-rep}),
\[%\begin{equation}\label{E-rep}
[\E^{-it\cK_0} \chi(\cK_0^2)](x,y)=\frac{1}{2\pi}\int_{-\infty}^{\infty}{\mathcal M}_t(k)\,
\E^{i|y-x|k}\chi(k^2+m^2)\,dk.
\]%\end{equation}
Note, that for an integral operator $T$,
\[
\Vert T\Vert_{L^1\to L^{\infty}}=\sup\limits_{\Vert f\Vert_{L^\infty}=1,\Vert g\Vert_{L^1}=1}
\langle f, Tg\rangle= \sup\limits_{x,y}|[T](x,y)|.
\]
Hence, it suffices to prove that for $j=-1,0,1$ and $p=0,1$,
\begin{equation}\label{Free-low-deKG0cay}
\sup\limits_{x,y}|\int_{-\infty}^{\infty}e^{i\omega(k)t}\E^{i(y-x)k}\omega^j(k)k^{p}\chi(k^2+m^2)\,dk|\le Ct^{-1/2},\quad t\ge 1.
\end{equation}
Here  we denote $\omega(k):=\sqrt{k^2+m^2}$. The integral in \eqref{Free-low-deKG0cay} is 
an oscillatory integral with the  phase function $\phi(k)=\omega(k)+\frac{y-x}{t }k$, satisfying
\[
|\phi''(k)|=\frac{m^2}{\omega^3(k)}\ge C (m,\chi),
\quad (k^2 + m^2)\in{\rm supp}\,\chi.
\]
Then (\ref{WD2}) follows by the  van der Corput lemma \cite [page 332] {St}.
\\
{\bf Proof of Lemma \ref{lem2}}
\\
One has
\begin{equation}\label{E-rep1}
[\E^{-it\cK_0} \zeta(\cK_0^2)](x,y)=\frac{1}{2\pi}\int_{-\infty}^{\infty}{\mathcal M}_t(k)\,
\E^{i|y-x|k}\zeta(k^2+m^2)\,dk.
\end{equation}
Let  $B_j(t)$, $j=0,1,-1$  be  integral  operators with  integral kernels
$$
B_j(x,y,t)=B_j(x-y,t)=\int_{-\infty}^{\infty}\,\E^{i\omega(k)t}\E^{ik(x-y)}\zeta(k^2+m^2)\omega^j(k)\,dk.
$$
Thus,  (\ref{WD3})  is equivalent to
\begin{equation}\label{bs2}
\Vert B_0(t)\Vert_{L^2_{\sigma}\to L^2_{-\sigma}}
+\Vert B_1(t)\Vert_{H^1_{\sigma}\to L^2_{-\sigma}}+\Vert B_{-1}(t)\Vert_{L^2_{\sigma}\to H^1_{-\sigma}}
\le C(1+ t)^{-\sigma}\!, ~~\sigma>0
\end{equation}
by (\ref{M}) and  (\ref{E-rep1}). We also took into account that  
$\Vert B_0(t)\Vert_{H^1_{\sigma}\to H^1_{-\sigma}} =\Vert B_0(t)\Vert_{L^2_{\sigma}\to L^2_{-\sigma}}$, since
$[B_0(t)f]'(x)= [B_0(t)f'](x)$.

First we prove (\ref{bs2})  for integer  values of $\sigma$.  Namely, for $\ell=0,1,2,...$,
\begin{equation}\label{bs3}
\Vert B_0(t)\Vert_{L^2_{\ell}\to L^2_{-\ell}}
+\Vert B_1(t)\Vert_{H^1_{\ell}\to L^2_{-\ell}}+\Vert B_{-1}(t)\Vert_{L^2_{\ell}\to H^1_{-\ell}}
\le C(1+t)^{-\ell}\!,  \quad t\ge0.
\end{equation}
Denote by $\hat f$ the Fourier transform of $f$:
\[
\hat f(k)=\int \E^{-iky} f(y)dy.
\]
Then
\begin{equation}\label{Bj}
\!\!\!\!\![B_j(t)f](x)
=\int_{-\infty}^{\infty}\E^{i\omega(k)t}\,\E^{ikx}\zeta(k^2+m^2)\hat f(k)\omega^j(k)\,dk.
\end{equation}
In the case $\ell=0$, (\ref{Bj}) implies
$$
\Vert B_0(t)f\Vert_{L^2}=\Vert e^{i\omega(k)t}\zeta(k^2+m^2)\hat f(k)\Vert_{L^2}
\le C\Vert \hat f\Vert_{L^2} =C\Vert f\Vert_{L^2} 
$$
Further, using integration by parts, we obtain
\begin{eqnarray}\nonumber
[B_0(t)f](x)&\!\!=\!\!&\frac{i}{t}\int \E^{i\omega(k)t}\partial_k\Big(\E^{ikx}\,\frac{\omega(k)}{k}\zeta(k^2+m^2)\hat f(k)\Big)\,dk\\
\nonumber
&\!\!=\!\!&-\frac{x}{t}\int \E^{i\omega(k)t+ikx}\,\frac{\omega(k)}{k}\zeta(k^2+m^2)\hat f(k)\,dk\\
\nonumber
&\!\!+\!\!&\frac{i}{t}\int \E^{i\omega(k)t+ikx}\,\frac{\omega(k)}{k}\zeta(k^2+m^2)\partial_k\hat f(k)\,dk\\
\nonumber
&\!\!+\!\!&\frac{i}{t}\int \E^{i\omega(k)t+ikx}\hat f(k)\partial_k\Big(\frac{\omega(k)}{k}\zeta(k^2+m^2)\Big)\,dk\\
\label{in1}
&=&g_1(x,t)+g_2(x,t)+g_3(x,t),\quad t\ge 1.
\end{eqnarray}
Evidently,
$$
\Vert g_3(t)\Vert_{L^2}=\frac{1}{t}\Vert \E^{i\omega(k)t}\hat f(k)\partial_k\Big(\frac{\omega(k)}{k}\zeta(k^2+m^2)\Big)\Vert_{L^2}
\le \frac{C}{t}\Vert \hat f\Vert_{L^2} =\frac{C}{t}\Vert f\Vert_{L^2}. 
$$
Similarly,
$$
\Vert \langle x\rangle^{-1} g_1(x,t)\Vert_{L^2} \le \frac{C}{t}\Vert f\Vert_{L^2},\quad
\Vert g_2(x,t)\Vert_{L^2} \le \frac{C}{t}\Vert \partial_k\hat f(k)\Vert_{L^2}\le  \frac{C}{t}\Vert \langle x\rangle  f(x)\Vert_{L^2}.
$$
Hence, (\ref{bs3}) with $\ell=1$ for the first term follows.  For $\ell=2,3,...$, (\ref{bs3})  follows similarly by $\ell$- times integration by parts.
Note, that
$$
[B_{-1}(t)f]'(x) =\int_{-\infty}^{\infty}\E^{i\omega(k)t}\,\E^{ikx}\zeta(k^2+m^2)\hat f(k)\frac {ik}{\omega(k)}\,dk.
$$  
$$
[B_{1}(t)f]'(x) =\int_{-\infty}^{\infty}\E^{i\omega(k)t}\,\E^{ikx}\zeta(k^2+m^2)\widehat{ f'}(k)\frac {\omega(k)}{ik}\,dk.
$$   
Hence, the  remaining terms  in the left hand side of  (\ref{bs3})  are evaluated similarly. 
Finally,  (\ref{bs3})  implies  (\ref{bs2}) by a suitable version of Stein-Weiss interpolation theorem: 
\begin{lemma}\label {RT} (cf. \cite[Theorem 5.4.1]{BL}).
Let $\sigma_0\ne \sigma_1$, and  let
$T: L^2_{\sigma_j}\to L^2_{-\sigma_j}$   with the norm $M_j=\sup\{\langle Tf,g\rangle: \Vert f\Vert_{L^2_{\sigma_j}}=\Vert g\Vert_{L^2_{\sigma_j}}=1\}$, $j=0,1$.
Then for $\sigma=\sigma(\theta)=(1-\theta)\sigma_0+\theta\sigma_1$, $\theta\in(0,1)$,  one has
$$
T: L^2_{\sigma}\to L^2_{-\sigma}  ~~{\rm with}~{\rm the}~{\rm norm}~~ M\le M_0^{1-\theta}M_1^{\theta}.
$$
\end{lemma}
For the convenience of readers we give the  proof in the Appendix A.
%%%%%%%%%%%%%%%%%%%%%%%%%%%%%%%%%%%%%%%%%%%%%%%%%%%%%%%%%%%%%%%%%%%%%%%%%%%%%%%%
\section{ Scattering properties of Schr\"odinger operator}\label{scat}
%%%%%%%%%%%%%%%%%%%%%%%%%%%%%%%%%%%%%%%%%%%%%%%%%%%%%%%%%%%%%%%%%%%%%%%%%%%%%%%%
Next we recall a few facts from scattering theory \cite{DT}, \cite{EKTM} of the Schr\"odinger operator $H=-(\partial^2_x+V)$.
Under the assumption
$V\in L^1_1$ there exist Jost solutions $f_\pm(x,k)=\E^{\pm ikx} h_\pm(x,k)$ of
\[
Hf=k^2f,\quad  k\in \overline{\C_+},
\]
normalized according to
\[
h_\pm(x,k)\sim  1,\quad x\to \pm \infty.
\]
Denote by  $\mathcal{A}$ the Banach algebra  of Fourier transforms of integrable functions
\[
\mathcal{A} = \left\{f(k):\,
f(k) = \int_\R \E^{ik p}\hat{f}(p)dp, \,\hat{f}(\cdot)\in L^1(\R) \right\}
\]
with the norm $\|f\|_{\mathcal{A}}= \|\hat{f}\|_{L^1}$, and by 
$\mathcal{A}_1$  the corresponding unital Banach algebra 
\[
\mathcal{A}_1 = \left\{f(k):\,
f(k) =c+ \int_\R \E^{ik p}\hat{g}(p)dp, \,\hat{g}(\cdot)\in L^1(\R),\,c\in\C\right\}
\]
with the norm $\|f\|_{\mathcal{A}_1}= |c|+\|\hat{g}\|_{L^1}$.
Evidently,  $\mathcal{A}$ is a subalgebra of $\mathcal{A}_1$.
Then
\begin{equation} \label{alg1}
h_\pm(x,\cdot) - 1,~~ h'_\pm(x,\cdot)\in\mathcal A,
\quad \forall x\in\R.
\end{equation}
%\begin{equation}\label{est31}
%\sup\limits_{k} |h_\pm(x,k)|\leq C_\pm,\quad\mbox{for}\quad\pm x\geq 0.
%\end{equation}
Let
\[
W(\varphi(x,k),\psi(x,k))=\varphi(x,k)\psi'(x,k)-\varphi'(x,k)\psi(x,k)
\]
be the usual Wronskian, and set
\[
W(k)=W(f_-(x,k),f_+(x,k)),\qquad W_\pm(k)=W(f_{\mp}(x,k),f_{\pm}(x,- k)).
\]
\[
T(k)= \frac{2i k}{W(k)},\quad R_\pm(k)= \mp\frac{W_\pm(k)}{W(k)},
\]
Recall that $T$ and $R_\pm$ are the entries of the scattering matrix, that is, the transmission and reflection coefficients.
%%%%%%%%%%%%%%%%%%%%%%%%%%%%%%%%%%%%%%%%%%%%%%%%%%%%%%%%%%%%%%%%%%%%%%%%%%%%%
For $V\in L^1_1$,  
$$
T(k)-1\in\mathcal{A},\qquad R_\pm(k)\in\mathcal{A}
$$
by  \cite[Theorem 2.1]{EKTM}.  Introduce the function
\begin{equation}\label{psi}
\psi(x,y,k)= h_+(y,k)h_-(x,k) T(k), \qquad y\ge x,
\end{equation}
and $\psi(x,y,k)=\psi(y,x,k)$ for $y<x$.  
Then the kernel
of the resolvent $R(\omega)=(H-\omega)^{-1}$ for $\omega=k^2\pm i0$, $k>0$, can be expressed as 
\begin{equation}\label{resolvent}
[R(k^2\pm i0)](x,y) %= -E-rep \frac{f_+(y,\pm k) f_-(x,\pm k)}{W(\pm k)}
=\mp\frac{\E^{\pm ik(y-x)}\psi(x,y,\pm k)}{2ik}.
\end{equation}
%%%%%%%%%%%%%%%
\begin{lemma}\label{lemconst}(cf. \cite[Lemma 2.3]{EKTM}) Let $V\in L^1_1$.
Then  $\psi(x,y,\cdot)\in\mathcal{A}_1$, and the following estimate holds
\begin{equation}\label{thh}
\| \psi(x,y, \cdot)\|_{\mathcal{A}_1} \le C,
\end{equation}
with some constant $C$, which does not depend on $x$ and $y$.
\end{lemma}
%%%%%%%%%%%%%%%
Finally, in \cite{EKTM} has been proved that for $V\in L^1_2$,
\begin{equation}\label{psi-est}
\Vert \xi(k)\partial_k\psi(x,y,k)\Vert_{\mathcal A}\le C(1+|x|)(1+|y|),
\end{equation}
where $\xi(k)$ is a smooth function such that $\xi(k)=0$ for $|k|\le 1$ and
$\xi(k)=1$ for $|k\ge 2$.
%%%%%%%%%%%%%%%%%%%%%%%%%%%%%%%%%%%%%%%%%%%%%%%%%%%%%%%%%%%%%%%%%%%%%%%%%%%%%%%%
\section{Perturbed Klein-Gordon equation}\label{KGP}
%%%%%%%%%%%%%%%%%%%%%%%%%%%%%%%%%%%%%%%%%%%%%%%%%%%%%%%%%%%%%%%%%%%%%%%%%%%%%%%%
The resolvent $\cR(\omega)$ of the operator \eqref{H} associated with the Klein-Gordon equation \eqref{KGE} can be expressed 
in terms of the resolvent   $R(\omega)$ of the Schr\"odinger operator $H$ as
\begin{equation}\label{R}
 \cR(\omega)=
   \begin{pmatrix}
 0            &  0\\
  -i          &  0
\end{pmatrix}
+
\begin{pmatrix}
 \omega &  i \\
  -i \omega^2&  \omega
\end{pmatrix} R(\omega^2-m^2).
\end{equation}
For the one-parameter group $\E^{-it\cK}$ of \eqref{KGEv} and for any $\chi \in C_0^\infty$ 
the spectral representations of type \eqref{Rep0} -\eqref{K0-rep} hold (cf. \cite [Formula (5.4)]{EKTM}):
\begin{eqnarray}\nonumber
\E^{- it\cK}P_c\,\chi (\cK^2)\, 
&=&\frac 1{2\pi i}\int\limits_\Gamma  \E^{-it\omega}  ( \cR(\omega+i0)- \cR(\omega-i0))\chi(\omega^2)\,d\omega\\
\label{E-rep}
&=&\frac{1}{2\pi}\int\limits_{\R}{\mathcal M}_t(k)\,\E^{i|y-x|k}\psi(x,y,k) \chi(k^2+m^2)dk,
\end{eqnarray}
where the functions ${\mathcal M}_t(k)$ and $\psi(x,y,k)$ are defined by \eqref{M} and \eqref{psi}.
%%%%%%%%%%%%%%%%%%%%%%%%%%%
Representation (\ref{E-rep}), Lemma \ref{thh}, and the appropriate version of the  van der Corput lemma 
\cite[Lemma 5.4]{EKTM} imply   
%%%%%%%%%%%%%%%%%%%%%%%%%%%%%%%%%%%%%%%%%%%%%%%%%%%%%%%%%%%%%%%%%%%%%%%%%%%%%%%%
\begin{theorem}\label{thKG1}  
i) Assume $V\in L^1_1$. Then for any smooth function $\chi$  with bounded support the following decay holds
\[
\big\| \E^{-it\cK}P_c\,\chi(\cK^2)\big\|_{\cL^1\to \cL^\infty}
={\mathcal O}(t^{-1/2}),\quad t\to \infty.
\]
ii)
Assume $V\in L^1_2$. Then in non-resonant case  the following decay holds
\[
\big\| \E^{-it\cK}P_c\,\chi(\cK^2)\big\|_{ \cL^1_1\to \cL^\infty_{-1}}
={\mathcal O}(t^{-3/2}),\quad t\to \infty.
\]
\end{theorem}
The detailed proof of the theorem can be found in \cite{EKTM}. The theorem immediately implies
\begin{theorem}\label{thKG2}  
i) Assume $V\in L^1_1$. Then for any smooth function $\chi$ with bounded support and any $\sigma>1/2$ the following decay holds
\[
\big\| \E^{-it\cK}P_c\,\chi(\cK^2)\big\|_{\cF_\sigma\to \cF_{-\sigma}}
={\mathcal O}(t^{-1/2}),\quad t\to \infty.
\]
ii)
Assume $V\in L^1_2$. Then in non-resonant case  the following decay holds
\[
\big\| \E^{-it\cK}P_c\,\chi(\cK^2)\big\|_{\cF_\sigma\to \cF_{-\sigma}}
={\mathcal O}(t^{-3/2}),\quad t\to \infty,\quad \sigma>3/2.
\]
\end{theorem}
Now we prove a high energy decay.  
\begin{theorem}\label{thKG3}
i) Let condition \eqref{V} holds with some $\beta>2$ and let  $\zeta(x)$ be a smooth function such that $\zeta(x)=0$ for $x\le m^2+1$ and
$\zeta(x)=1$ for $x\geq m^2 + 2$.
Then  for any $\sigma>3/4$, the decay holds  
\begin{equation}\label{b1}
\big\|\E^{-it\cK}\,\zeta(\cK^2)\big\|_{{\mathcal F}_\sigma\to {\mathcal F}_{-\sigma}}
={\mathcal O}(t^{-1/2}),\quad t\to \infty.
\end{equation}
ii)
Let condition \eqref{V} holds with some $\beta>3$.  Then for any $\sigma>3/2$,
\begin{equation}\label{b2}
\big\|\E^{-it\cK}\,\zeta(\cK^2)\big\|_{{\mathcal F}_\sigma\to {\mathcal F}_{-\sigma}}
={\mathcal O}(t^{-3/2}),\quad t\to \infty.
\end{equation}
\end{theorem}
%%%%%%%%%%%
\begin{proof}
{\it Step i)}.
Substituting the  resolvent identity 
$$
R(\lambda)= R_0(\lambda)- R_0(\lambda)V R_0(\lambda)+ (R_0(\lambda)V)^2 R_0(\lambda) + (R_0(\lambda)V)^3 R(\lambda)
$$ 
into \eqref{R}, we obtain
\[
\E^{-it\cK}\,\zeta(\cK^2)=\cU_0(t)+\cU_1(t)+\cU_2(t)+{\mathcal W}(t),,
\]
where $\cU_0(t)=\E^{-it\mathbf{K_0}}\,\zeta(\cK_0^2)$, and
\begin{equation*}
\cU_j(t)=\frac 1{2\pi i}\int\limits_{\Gamma}\! \E^{-i\omega t}
([(\cR_0{\mathcal V})^j\cR](\omega+i0)-[(\cR_0{\mathcal V})^j\cR_0](\omega-i0))\zeta(\omega^2) d\omega, \quad j=1,2,
\end{equation*}
\begin{equation*}
{\mathcal W}(t)=\frac 1{2\pi i}\int\limits_{\Gamma}\! \E^{-i\omega t}
([(\cR_0{\mathcal V})^3\cR](\omega+i0)-[(\cR_0{\mathcal V})^3\cR](\omega-i0))\zeta(\omega^2) d\omega.
\end{equation*}
Here
\begin{equation}\label{cV}
{\mathcal V}(x)=\left(\begin{array}{cc}
  0               &   0
  \\
  iV(x)   &   0
  \end{array}\right).
\end{equation}
Due to Lemma \ref{lem2}, 
\begin{equation}\label{U0}
\big\| \cU_0(t)\big\|_{{\mathcal F}_{\sigma}\to {\mathcal F}_{\sigma}}
={\mathcal O}(t^{-\sigma}),\quad t\to \infty,\quad \sigma>0.
\end{equation}
Let us prove that for $\beta>2$ %and $\sigma=\beta/2$
\begin{equation}\label{U1}
\big\| \cU_j(t)\big\|_{{\mathcal F}_{\beta/2}\to {\mathcal F}_{-\beta/2}}
={\mathcal O}(t^{-\beta/2}),\quad t\to \infty,\quad j=1,2.
\end{equation}
\smallskip\\
For arbitrary  $\Psi_0\in {\mathcal F}_{\beta/2}$, denote  $\Psi_{j+1}(t)=\cU_j(t)\Psi_0$, $j=0,1,2$.
%%%%%%%%%%%%%%%%%%
\begin{lemma}(cf. \cite[Lemma 3.8]{KK10}, \cite[Lemma 36.6]{KK12}). 
The convolution representation hold
\begin{equation}\label{conv}
\Psi_{j+1}(t)=
 i\int_0^t \cU_0(t-\tau){\mathcal V} \Psi_j(\tau)~d\tau,~~~~t\in\R,\quad j=1,2,
\end{equation}
where the integral converges in ${\mathcal F}_{-\beta/2}$.
\end{lemma}
By Lemma \ref{lem2},
$\Vert\Psi_1\Vert_{\cF_{-\beta/2}}\le C (1+|t|)^{-\beta/2}\Vert \Psi_0\Vert_{\cF_{\beta/2}}$.
 Applying  Lemma \ref{lem2} to the integrand in (\ref{conv}), 
 we obtain for $\sigma=\beta/2$
 \begin{equation*}
 \Vert \cU_0(t-\tau){\mathcal V} \Psi_1(\tau)\Vert_{\cF_{-\sigma}}\le 
 \frac{C\Vert {\mathcal V} \Psi_1(\tau)\Vert_{\cF_{\sigma}}}{(1\!+|t-\!\tau|)^{\sigma}}
 \le\frac{C_1\Vert \Psi_1(\tau)\Vert_{\cF_{-\sigma}}}{(1\!+|t-\!\tau|)^{\sigma}}
  \le\!\frac{C_2\Vert \Psi_0\Vert_{\cF_{\sigma}}}{(1\!+|t-\!\tau|)^{\sigma}(1\!+|\tau|)^{\sigma}}
\end{equation*}  
Therefore, integrating here in $\tau$, we obtain 
\begin{equation*}
  \Vert \Psi_2(t)\Vert_{\cF_{-\beta/2}} \le C (1+|t|)^{-\beta/2}\Vert \Psi_0\Vert_{\cF_{\beta/2}}.
\end{equation*}
Similarly,
\begin{equation*}
  \Vert \Psi_3(t)\Vert_{\cF_{-\beta/2}} \le C (1+|t|)^{-\beta/2}\Vert \Psi_0\Vert_{\cF_{\beta/2}}.
\end{equation*}
Hence, \eqref{U1} follows.
\smallskip\\
It remains to estimate the operator  ${\mathcal W}(t)$ with the kernel
\begin{equation}\label{W}
\!\!\!\!\!{\mathcal W}(t,x,y)
=\!\int\limits_{\R^3} \frac{{\mathbf V}(u,\rm w,z)}{16 \pi i}\left(\int\limits_\R\! \xi(k){\mathcal M}_t(k)
\frac{\E^{ikp(x,y,u,\rm w,z)}}{k^3}\psi(y,z,k)dk\right) dud{\rm w} dz,
\end{equation}
where we denote  $p(x,y,u,\rm w,z)=|x-u|+|u-{\rm w}|+|{\rm w}-z|+|z-y|$,  $\xi(k)=\zeta(k^2+m^2)$,
${\mathbf V}(u,\rm w,z)=V(u)V(\rm w)V(z)$.
We will apply the following version of \cite[Lemma 2]{MSW}:
%%%%%%%%%%%%%%%%%%%%%%%%%%%%%%%%%%%%%%%%%%%%%%%%%%%%%%%%%%%%%%%%%%%%%%%%%%%%%%%%%
%%%%%%%%%%%%%%%%%%%%%%%%%%%%%%%%%%%%%%%%%%%%%%%%%%%%%%%%%%%%%%%%%%%%%%%%%%%%%%%%%
\begin{lemma}\label{l1} (cf. \cite[Lemma 5.5]{EKTM})
Let $\eta(k)$, $k\geq 1$, be a smooth function such that
 $|\eta^{(j)}(k)|\le k^{-j}$ for $j=0,1$. Then for any $g(k)\in \mathcal A_1$,
 $\alpha> 3/2$ and $t\ge 1$
\begin{equation}\label{OI-est}
\sup\limits_{p\in\R}\Big|\int_{1}^\infty\eta(k)\frac{\E^{\pm it\omega(k)+ ik p}}
{k^\alpha}g(k)dk\Big|\le C t^{-1/2}\Vert g\Vert_{\mathcal A_1} ,\quad \omega(k)=\sqrt{k^2+m^2}.
\end{equation}
\end{lemma}
Applying this lemma with  $g(k)=\psi(y,z,k)$, $p= p(x,u,\rm w,z,y)$, and taking into account \eqref{thh}, we get
\begin{equation}\label{born11}
\sup\limits_{x,y}|{\mathcal W}^{ij}(t,x,y)|\le C t^{-1/2}\|V\|_{L^1}^3\le C_1t^{-1/2},\quad t\ge 1,\quad i,j=1,2.
\end{equation}
Here ${\mathcal W}^{ij}$ are the entries of the matrix ${\mathcal W}$. Further,
\begin{eqnarray}\nonumber
\!\!\!\!\!\!\!\!\!\!\!\! &&\!\!\!\!\!\!\!\!\partial_x{\mathcal W}(t,x,y)\\
\nonumber
\!\!\!\!\!\!\!\!\!\!\!\!&&\!\!\!\!\!\!\!\!=\int\limits_{\R^3} {\rm sgn} (x\!-\!u)\frac{{\mathbf V}(u,\rm w,z)}{16 \pi}\left(\int\limits_{\R} \xi(k)
{\mathcal M}_t(k)
\frac{\E^{ikp(x,y,u,\rm w,z)}}{k^2}\psi(y,z,k)dk\right) dud{\rm w} dz.
\end{eqnarray}
Hence, similarly to \eqref{born11},
\begin{equation}\label{born12}
\sup\limits_{x,y}|\partial_x{\mathcal W}^{1j}(t,x,y)|\le C t^{-1/2},\quad t\ge 1,\quad j=1,2.
\end{equation}
Bounds \eqref{born11} and \eqref{born12} imply that
\[%\begin{equation}\label{U2}
\Vert {\mathcal W}(t)\Vert_{\cL^{!}\to \cL^\infty}\le C t^{-1/2}, \quad t\ge 1.
\]%\end{equation}
Together with  \eqref{U0} and \eqref{U1} this gives \eqref{b1}.
\smallskip\\
{\it Step ii)}.
Now we rewrite \eqref{W} as follows
\begin{eqnarray}\nonumber
\!\!\!\!\!\!\!\!\!\!&&\!\!\!\!\!\!\!\!\!\!\!\!\!{\mathcal W}(t,x,y)\\
\label{W1}
\!\!\!\!\!\!\!\!\!\!&&\!\!\!\!\!\!\!\!\!\!\!\!\!
= \!\sum_{\pm} \!\int\limits_{\R^3} \frac{{\mathbf V}(u,{\rm w},z)}{16 \pi i}\left(\int\limits_\R\! \xi(k)A_{\pm}(k)
\frac{\E^{it\omega(k)+ikp(x,y,u,{\rm w},z)}}{k^3}\, \psi(y,z,k)dk\!\right)\! dud{\rm w} dz,
\end{eqnarray}
where
\[
A_{\pm}(k)=\begin{pmatrix}
1  & \mp\ds\frac{i}{\omega(k)}\\
  \pm i\omega(k)& 1
\end{pmatrix},
\]
Integrating by parts in inner integral of \eqref{W1}, we get
\begin{eqnarray}\nonumber
&&F_{\pm}(t,x,y,u,{\rm w},z):=\int \xi(k)A_{\pm}(k)\frac{\E^{it\omega(k)+ikp(x,y,u,{\rm w},z)}}{k^3}\, \psi(y,z,k)dk\\
\nonumber
&&=\frac{i}{t}\int \E^{it\omega(k)}\frac{\partial}{\partial k}
\Big[\E^{ikp(x,y,u,{\rm w},z)}\xi(k)\frac{\omega(k)}{k^4}\,A_{\pm}(k) \psi(y,z,k)\Big]dk.
\end{eqnarray}
Recall that in the case $V\in L^1_2$,
$\Vert \frac{\partial}{\partial k}\psi(z,y,k)\Vert_{\mathcal A}\le C(1+|z|)(1+|y|)$  by \eqref{psi-est}. 
Moreover,
\[
|x-u|+|u-{\rm w}|+|{\rm w}-z|+|z-y|\le (1+|x|)(1+|y|)(1+2|z|) (1+2|u|)(1+2|{\rm w}|).
\] 
Hence, by Lemma \ref{l1}, 
\[
|F_{\pm}(t,x,y,u,{\rm w},z) |\le Ct^{-3/2} (1+|x|)(1+|y|)(1+|u|)(1+|{\rm w}|)(1+|z|),\quad t\ge 1.
\]
The last estimate and \eqref{W1}  then imply
\[
|{\mathcal W}^{ij}(x,y)|\le  Ct^{-3/2} (1+|x|)(1+|y|),\quad t\ge 1,\quad i,j=1,2.
\] 
Similarly,
\[
|\partial_x {\mathcal W}^{1j}(x,y)|\le  Ct^{-3/2} (1+|x|)(1+|y|),\quad t\ge 1,\quad j=1,2.
\]
Therefore,
\[
\Vert {\mathcal W}(t)\Vert_{\cL^1_1\to\cL^\infty_1}\le  Ct^{-3/2}, \quad t\ge 1.
\]
Together with  \eqref{U0} and \eqref{U1} this gives \eqref{b2}.
\end{proof}
%%%%%%%%%%%%%%%%%%%%
\appendix
\section{Proof of Lemma \ref{RT}}
%%%%%%%%%%%%%%%%%%%%
%%%%%%%%%%%%%%%%%%%%%%%%%%%%%%%%%%%%%%%%%%%%%%%%%%%%%%%%%%%%%%%%%%%%%
Suppose that $f$, $g$ are  simple functions, and for $0\le {\rm Re}\, z\le 1$ define the functions 
$$
\varphi(z):=\varphi(x,z)=\langle x\rangle^{-\sigma+\sigma(z)}f(x),
\quad  \psi(z):=\psi(x,z)=\langle x\rangle^{\sigma-\sigma(z)}g(x),
$$
where $\sigma(z):=(1-z)\sigma_0+z\sigma_1$. One has
$$
\Vert\varphi(it)\Vert_{L^2_{-\sigma_0}}^2=
\int \langle x\rangle^{-\sigma_0} |\langle x\rangle^{-\sigma+2(1-it)\sigma_0+it\sigma_1)}f(x)|^2dx
=\Vert f\Vert_{L^2_{-\sigma}}^2=1,
$$
$$
\Vert\varphi(1+it)\Vert_{L^2_{-\sigma_1}}^2=
\int \langle x\rangle^{-\sigma_1} |\langle x\rangle^{-\sigma-2it\sigma_0+(1+it)\sigma_1)}f(x)|^2dx
=\Vert f\Vert_{L^2_{-\sigma}}^2=1.
$$
Similarly,
$\Vert\psi(it)\Vert_{L^2_{\sigma_0}}=\Vert\psi(1+it)\Vert_{L^2_{\sigma_1}}=1$.
The function
$F(z)=\langle T\varphi(z),\psi(z)\rangle$  is  analytic  in $0<  {\rm Re}\, z<1$, and  bounded  and continuous  in  $0\le  {\rm Re}\, z\le 1$.
Moreover,
\begin{eqnarray}\nonumber
|F(it)|&\le& \Vert T (\varphi(it))\Vert_ {L^2_{-\sigma_0}}\Vert\psi(it)\Vert_{L^2_{\sigma_0}}\le M_0,\\
\nonumber
|F(1+it)|&\le& \Vert T (\varphi(1+it))\Vert_ {L^2_{-\sigma_1}}\Vert\psi(1+it)\Vert_{L^2_{\sigma_1}}\le M_1.
\end{eqnarray}
Besides, $\varphi(\theta)=f$, $\psi(\theta)=g$, and hence $F(\theta)=\langle Tf,g\rangle$.
Applying  the Hadamard three-lines theorem (cf. \cite [Lemma 1.1.2]{BL}), we obtain
$$
M:=\sup\{\langle Tf,g\rangle: \Vert f\Vert_{L^2_{\sigma}}=\Vert g\Vert_{L^2_{\sigma}}=1\}\le M_0^{1-\theta}M_1^{\theta}.
$$
%%%%%%%%%%%%%%%%%%%%%%%%%%%%%%%%%%%%%%%%%%%%%%%%%%%%%%%%%%%%%%%%%%%%%

\end{document}